\documentclass[letterpaper,11pt]{amsart}
\title{Gaussian Conditional Independence Relations Have No Finite Complete Characterization}
\author{Seth Sullivant}
\address{Department of Mathematics and Society of Fellows, Harvard University, Cambridge, MA 02138}

\usepackage{latexsym,array,delarray,amsthm,amssymb,epsfig}%eufrak
\theoremstyle{plain}
\newtheorem{thm}{Theorem}[section]
\newtheorem{lemma}[thm]{Lemma}
\newtheorem{prop}[thm]{Proposition}

\theoremstyle{definition}
\newtheorem{defn}[thm]{Definition}

\theoremstyle{remark}

\newcommand{\zz}{\mathbb{Z}}

\newcommand{\rr}{\mathbb{R}}

\newcommand{\kk}{\mathbb{K}}
\newcommand{\bbe}{\mathbb{E}}

\newcommand{\bfa}{\mathbf{a}}

\newcommand{\bfu}{\mathbf{u}}
\newcommand{\bfv}{\mathbf{v}}

\newcommand{\bfx}{\mathbf{x}}

\newcommand{\cn}{\mathcal{N}}

\newcommand{\ind}{\mbox{$\perp \kern-5.5pt \perp$}}

\begin{document}

\begin{abstract}
We show that there can be no finite list of conditional independence relations which can be used to deduce all conditional independence implications among Gaussian random variables.  To do this, we construct, for each $n> 3$ a family of $n$ conditional independence statements on $n$ random variables which together imply that $X_1 \ind X_2$, and such that no subset have this same implication.  The proof relies on binomial primary decomposition.
\end{abstract}

\maketitle

%%%%%%%%%%%%%%%%%%%%%%%%%%%%%%%%
%%%%%%%%%%%%%%%%%%%%%%%%%%%%%%%%
%%%%%%%%%%%%%%%%%%%%%%%%%%%%%%%%
%%%%%%%%%%%%%%%%%%%%%%%%%%%%%%%%

\section{Introduction}

A fundamental question about the nature of conditional independence is whether or not there exists a finite set of conditional independence relations, from which all general conditional independence implications can be deduced.  This problem was resolved in the negative by Studen\'{y} \cite{Studeny1992}, who showed that there can be no such finite characterization.  To prove this, he exhibited infinite families of conditional independence implications, using the multi-information function and properties of the submodular cone.

Studen\'{y}'s work leaves open the question of whether or not it is possible to find such a finite axiom characterization over restricted classes of random variables.  For instance, if we assume that all random variables are binary, there are more independence relations that hold.   Perhaps it is possible that among the many new independence relations there can be a finite axiom system. 

In this note, we show that there is no finite characterization in the regular Gaussian case, by exhibiting an infinite family of conditional independence implications which cannot be deduced from any other conditional independence implications.  The main result of this note is:

\begin{thm} \label{thm:big}
Let $X \sim \mathcal{N}( \mu, \Sigma)$ with $\Sigma$ positive definite.  Suppose that $X$ satisfies the conditional independence constraints:
$$X_1 \ind X_2 | X_3,  X_2 \ind X_3 | X_4,  \ldots, X_{n-1} \ind X_n | X_1, X_n \ind X_1 | X_2$$
with $n \geq 4$ .  Then $X$ also satisfies the marginal independence constraints:
$$X_1 \ind X_2, X_2 \ind X_3, \ldots, X_{n-1} \ind X_n, X_1 \ind X_n.$$
However, if $X$ satisfies only a subset of the first CI constraints, it may not satisfy any of the second set.
\end{thm}

The proof relies on three simple ideas.  First, we use the fact that conditional independence in the Gaussian case is an algebraic restriction on the covariance matrix $\Sigma$.  So our collection of conditional independence constraints gives an ideal $I_n$ in $\rr[\Sigma]$.  Next, we compute the primary decomposition of $I_n$.  Since, in our case, $I_n$ will be a binomial ideal, we can exploit the results of \cite{Hosten2000} to determine the minimal primes of $I_n$.  Then we use Hadamard products (a trick we learned from \cite{Matus2002}) to show that one of these components does not intersect the cone of positive definite matrices and that the other component yields the desired conditional independence implications.

%%%%%%%%%%%%%%%%%%%%%%%%%%
%%%%%%%%%%%%%%%%%%%%%%%%%%
%%%%%%%%%%%%%%%%%%%%%%%%%%
%%%%%%%%%%%%%%%%%%%%%%%%%%

\section{The Proof}

Recall that a regular multivariate Gaussian $X \sim \cn(\mu, \Sigma)$ is complete specified by its mean vector $\mu \in \rr^n$ and its symmetric positive definite covariance matrix $\Sigma \in PD_n$.
First of all, we show that conditional independence corresponds to an algebraic constraint on the covariance matrix $\Sigma$.

\begin{prop}
Let $A, B, C$ be disjoint subsets of $[n]$.  Then the joint distribution for $X$ satisfies the conditional independence constraint $A \ind B | C$ if and only if the submatrix
$\Sigma_{A \cup C, B \cup C}$ has rank less than or equal to $\#C$.
\end{prop}

\begin{proof}
If $X  \sim  \mathcal{N}(\mu, \Sigma)$ follows a multivariate Gaussian distribution, then the conditional distribution of $X_{A \cup B}  |  X_C = x_c$ is also Gaussian with distribution 
$$ \mathcal{N}\left( \mu_{A\cup B} +  \Sigma_{A\cup B,C}  \Sigma_{C,C}^{-1}( x_C - \mu_C),   \Sigma_{A \cup B,A \cup B} - \Sigma_{A\cup B,C} \Sigma_{C,C}^{-1}  \Sigma_{C,A \cup B} \right),$$
(see, for example, \cite[\S B.6]{Bickel2001}).  The conditional independence statement $A \ind B | C$ holds if and only if  $(\Sigma_{A \cup B,A \cup B} - \Sigma_{A\cup B,C} \Sigma_{C,C}^{-1}  \Sigma_{C,A \cup B})_{A,B}  =  0$.  The $A,B$ submatrix of $\Sigma_{A \cup B,A \cup B} - \Sigma_{A\cup B,C} \Sigma_{C,C}^{-1}  \Sigma_{C,A \cup B}$ is easily seen to be $\Sigma_{A,B}  -  \Sigma_{A,C}  \Sigma_{C,C}^{-1}  \Sigma_{C,B}$ which is the Schur complement of the matrix
$$
\Sigma_{A \cup C,  B \cup C}  =  
\begin{pmatrix}
\Sigma_{A,B}  &   \Sigma_{A,C}  \\
\Sigma_{C,B}  &   \Sigma_{C,C}
\end{pmatrix}.
$$
Since $\Sigma_{C,C}$ is always invertible (it is positive definite), the Schur complement is zero if and only if the matrix $\Sigma_{A \cup C,  B \cup C}$ has rank less than or equal to $\#C$.
\end{proof}

Thus, if $\mathcal{M}  =  \{A_1 \ind B_1 | C_1, \ldots,  A_m \ind B_m | C_m \}$ is a conditional independence model, we naturally get the \emph{conditional independence ideal}
$$CI_\mathcal{M}   =   \left<   \#C+1 \mbox{ minors of }  \Sigma_{A \cup C, B \cup C}   \, \, | \,\, A \ind B | C \in \mathcal{M}  \right>  \subseteq  \rr[\Sigma].$$
To determine which conditional independence statements a collection of independence statements imply, we investigate the primary decomposition of the ideal $CI_\mathcal{M}$, and figure out which components of $CI_\mathcal{M}$ intersect the positive definite cone $PD_n$.

Define $\mathcal{M}_n$ to be the cyclic system of conditional independence constraints
$$\mathcal{M}_n = \{1 \ind 2 | 3, 2 \ind 3 | 4, \ldots,  n-1 \ind n | 1,   n \ind 1 | 2 \}$$
and let $I_n  =  CI_{\mathcal{M}_n}$ be the ideal  defining this cyclic model.  Or goal is to compute the primary decomposition of this ideal.

%%%%%%%%%%%%%%%%%%%%%%%%%%
%%%%%%%%%%%%%%%%%%%%%%%%%%
%%%%%%%%%%%%%%%%%%%%%%%%%%
%%%%%%%%%%%%%%%%%%%%%%%%%%

\subsection{Minimal primes of lattice basis ideals}

Let $B$ be a finite collection of integral vectors that form a basis for a saturated lattice $\mathcal{L}  \subseteq \zz^n$.  The \emph{lattice basis ideal} associated to $B$ is the binomial ideal
$$I_B = \left<  \bfx^{\bfu^+}  -  \bfx^{\bfu^-} \, \,  | \, \,  \bfu \in B  \right> \subset \kk[\bfx],$$
where $\bfu = \bfu^+ -  \bfu^-$ is the cancellation free representation of $\bfu$ as the difference of two nonnegative integral vectors, and $\bfx^\bfa = x_1^{a_1}x_2^{a_2} \cdots x_n^{a_n}$ is the monomial vector notation.   A thorough study of binomial ideals was done in \cite{Eisenbud1996} and a simple combinatorial characterization of the minimal primes of lattice basis ideals was discovered in \cite{Hosten2000}.  Their combinatorial characterization will be of considerable use, because of the following:

\begin{prop}
The conditional independence ideal 
$$I_n  = \left<  \sigma_{33} \sigma_{12} - \sigma_{13} \sigma_{23},  
\sigma_{44} \sigma_{23} - \sigma_{24} \sigma_{34},  \ldots, 
\sigma_{11} \sigma_{n-1,n} - \sigma_{1n-1} \sigma_{n1},
\sigma_{22} \sigma_{1n} - \sigma_{2n} \sigma_{12}  \right>$$
is a lattice basis ideal.
\end{prop}

\begin{proof}
Let $\bfu_i$ denote the exponent vector attached the conditional independence statement $X_{i-2} \ind X_{i-1} | X_i$.  That is, $\bfu_i =  e_{ii} + e_{i-2,i-1} - e_{i-1,i} - e_{i-2,i}$, where the indices are interpreted modulo $n$.  These vectors are linearly independent because the vector $e_{ii}$ only appears in vector $\bfu_i$.  Furthermore, this shows that the vectors span a saturated lattice because the submatrix of the matrix $M_n = [\bfu_1; \ldots;  \bfu_n]$ with columns corresponding to the diagonal entries in $\Sigma$, is the identity matrix.
\end{proof}

Since $I_n$ is a lattice basis ideal, we can use the results from \cite{Hosten2000} to study the primary decomposition of $I_n$.   Let $B = \{\bfu_1, \ldots, \bfu_k\}$ be a basis for a saturated lattice.  We form the matrix $M \in \zz^{k \times n}$ whose rows are the vectors $\bfu_i$.  
From \cite{Eisenbud1996}, it is known that the minimal primes of a lattice basis ideal are completely determined by the indeterminates that appear in them.  In particular, let $S \subseteq [m]$ be a collection indexing indeterminates.  Then the prime ideal associated to this subset is the ideal:

$$I_S =  \left<  x_i  \, \, | \, \, i \in S  \right>   +   \left<  \bfx^{\bfu^+} -  \bfx^{\bfu^-}  \, \, | \,\, \bfu \in B, \, \, {\rm supp}(\bfu)  \cap S = \emptyset \right>  :  \prod_{i \notin S}  x_i^\infty.$$
Thus, one must determine the sets $S$ that give minimal primes.  When $S = \emptyset$, $I_S = I_B : \prod  x_i^\infty$ is the toric ideal associated to the lattice $\mathcal{L}$.  This toric ideal is always a minimal prime of $I_B$.  The main result of \cite{Hosten2000} is that the other minimal primes of the lattice basis ideal $I_B$ can be read off from the sign patterns in the matrix $M$.

A matrix $M$ is called \emph{mixed} if every row contains a positive entry and a negative entry.  Let $S$ be a subset of $[m]$ indexing a possible minimal prime.  After permuting rows and columns of $M$, and relabeling $S$, we can assume that $S = \{1,2, \ldots, t\}$ and $M$ has the form
$$
M  = \left( \begin{array}{c|c}
N & B   \\ 
\hline 0 & D  \end{array}  \right),
$$
where $N$ has no all zero rows.

\begin{defn}
A matrix $N$ is called \emph{irreducible} if 
\begin{enumerate}
\item  $N$ is a mixed $s \times t$ matrix with $t \leq s$  and
\item one  cannot bring $N$ into the form 
$$N =  \left( \begin{array}{c|c} 
N' & B'  \\
\hline 0 & D'  \end{array}  \right)$$
by permuting rows and columns where $N'$ is a mixed $s' \times t'$ matrix with $t' \leq s'$ and $D'$ is a $(s - s') \times (t - t')$ matrix with $t - t' > s - s'$.
\end{enumerate}
\end{defn}

\begin{thm}\cite{Hosten2000}  \label{thm:irred}
A set of $S \subseteq [m]$ yields a minimal prime $I_S$ of the lattice basis ideal $I_B$ if and only if the associated matrix $N_S$ is an irreducible matrix.
\end{thm}

Now we will apply Theorem \ref{thm:irred} to determine the minimal primes of the ideals $I_n$.

\begin{lemma} \label{lem:minprime}
For $n \geq 4$, the minimal primes of $I_n$ are the toric ideal $I_{A_n} =  I_n :  \prod \sigma_{ij}^{\infty}$ and the monomial ideal $\left< \sigma_{12}, \sigma_{23}, \ldots, \sigma_{1n} \right>$.
\end{lemma}

\begin{proof}
For $n = 4$, this can be checked directly in Macaulay2 \cite{M2} or proven directly by a slight variation on the argument below.  So assume henceforth that $n \geq 5$.

The toric ideal $I_{A_n}$ is necessarily a minimal prime of $I_n$.  We need to show that the only set of variables $S$ that induces an irreducible decomposition of the lattice basis matrix is the one 
corresponding to the set of variables $\{\sigma_{12}, \sigma_{23}, \ldots, \sigma_{1n} \}$.  First, we want to describe the lattice basis matrix $M_n$, which we need to find irreducible decompositions of.  Since the binomials $\sigma_{ii} \sigma_{i-1,i-2} - \sigma_{i-1,i} \sigma_{i-2,i}$ only have index pairs $(i,j)$  with $|i-j| \leq 2 \mod n$, we only need to use $3n$ columns.   We separate these columns into three groups, depending on whether $|i - j| = 0,$ $1,$ or $2$.   For $n = 5$, $M_5$ is a $5 \times 15$ matrix:

$$M_5 =  \left( \begin{array}{ccccc|ccccc|ccccc}
+1 & 0 & 0 & 0 & 0 &+1 & -1 & 0 & 0 & 0 & -1 & 0 & 0 & 0 & 0 \\
0 & +1 & 0 & 0 & 0 & 0 & +1 & -1 & 0 & 0 & 0 & -1 & 0 & 0 & 0 \\
0 & 0 & +1 & 0 & 0 & 0 & 0 & +1 & -1 & 0 & 0 & 0 & -1 & 0 & 0  \\
0 & 0 & 0 & +1 & 0 & 0 & 0 & 0 & +1 & -1 & 0 & 0 & 0 & -1 & 0 \\
0 & 0 & 0 & 0 & +1 & -1 & 0 & 0 & 0 & +1 & 0  & 0 & 0 & 0 & -1  \\
\end{array}  \right)$$

In general, after reordering the columns in each block, we will have an identity matrix, a circulant matrix on the vector $(+1, -1, 0, \ldots, 0)$, and minus the identity matrix for the three blocks of size $n$.  Note that the central circulant block, corresponding to the indeterminates $\sigma_{12}, \sigma_{23}, \ldots, \sigma_{1n}$, is an irreducible submatrix as it is $n \times n$, mixed, and no nonempty submatrix of it is mixed.  This is the desired minimal prime we were seeking.  It remains to show that no other subsets of the variables induce an irreducible decomposition of the matrix $M$.  

To see why there are no other irreducible decompositions, we can just count the number of nonzero entries in each column.  Note that there are $\leq 2$ nonzero entries in each column.  So suppose that $N$ were an irreducible matrix arising from choosing some subsets $S$ of the variables.  Each row of $N$ must have at least one $+1$ and one $-1$ entry, since $N$ is mixed.  So if $N$ is an $s \times t$ matrix, the number of nonzero entries is $ \geq 2s$.  On the other hand, since each column has at most $2$ nonzero entries, we know that there are $\leq 2t$ nonzero entries in $N$.  Furthermore, since $N$ is irreducible  we have $s \geq t$, so such an $N$ can only exist when $s = t$ and every column used has two nonzero entries.  This implies that only the indeterminates corresponding to the circulant submatrix of $M$ can be among the variables associated to the minimal prime.  However, we have already seen that this set of variables yields an irreducible submatrix, thus there can be no other minimal primes besides the two we have already found.
\end{proof}

%%%%%%%%%%%%%%%%%%%%%%%%%%
%%%%%%%%%%%%%%%%%%%%%%%%%%
%%%%%%%%%%%%%%%%%%%%%%%%%%
%%%%%%%%%%%%%%%%%%%%%%%%%%

\subsection{Hadamard Products}

The next step in the proof depends on analyzing the toric ideal $I_{A_n}$, and showing that the variety $V(I_{A_n})$ does not intersect the positive definite cone.  Ultimately, the basic idea comes from a fact about Hadamard products of positive definite matrices.

\begin{lemma}
Let $\Sigma$ and $T$ be $n \times n$ positive definite matrices.  Then the Hadamard product $\Sigma \ast T$ defined by
$$(\Sigma \ast T)_{ij}  =  \sigma_{ij} \tau_{ij}$$
is also positive definite.
\end{lemma}

\begin{proof}
Let $X = (X_1, \ldots, X_n)$ and $Y = (Y_1, \ldots, Y_n)$ be random vectors both with mean $0$ and with covariance matrices $\Sigma$ and $T$ respectively, and suppose that $X \ind Y$.  Also, suppose that $X$ and $Y$ have $n$-dimensional support.  Let $X \ast Y$ be the new random vector $X\ast Y = (X_1 Y_1, X_2 Y_2, \ldots, X_n Y_n)$.  Since $X$ and $Y$ both have mean zero, we compute the  covariance $\Gamma$ of $X \ast Y$ as
$$\gamma_{ij} =   \bbe[ X_i Y_i  X_j Y_j]  =  \bbe [X_i X_j]  \bbe[Y_i Y_j]  =  \sigma_{ij} \tau_{ij}.$$
Thus, $\Gamma =  \Sigma \ast T$.  Since $X$ and $Y$ have full dimensional support, so does $X \ast Y$ and thus $\Gamma$, the covariance matrix of a random variable with full dimensional support, must be a positive definite matrix. 
\end{proof}

\begin{lemma}\label{lem:badcomp}
The variety $V(I_{A_n})$ does not intersect the positive definite cone.
\end{lemma}

\begin{proof}
Given a permutation $\pi \in S_n$ and an $n \times n$ matrix $\Sigma$, let $\pi(\Sigma)$ be the matrix obtained by simultaneously permuting rows and columns by $\pi$.  Let $\pi$ denote the cycle $(12\cdots n)$ and let $\Sigma$ be a generic matrix.  Let the matrix $\Sigma_n$ be defined by the repeated Hadamard product:
$$\Sigma_n  =  \Sigma \ast \pi (\Sigma) \ast \pi^2(\Sigma) \ast  \cdots  \ast \pi^{n-2}( \Sigma)  \ast \pi^{n-1} (\Sigma).$$
If $\Sigma$ is positive definite, then so is $\Sigma_n$.  In particular, 
$$\det \left( (\Sigma_n) _{13,13} \right) > 0.$$
This inequality is equivalent to
$$\prod_{i = 1}^n  \sigma_{ii}^2  >  \prod_{i =1}^n  \sigma_{i-2,i}^2.$$
On the other hand, the lattice $\ker_\zz A_n$ is spanned by the vectors $\bfu_i =  e_{ii} + e_{i-2,i-1} - e_{i-1,i} - e_{i-2,i}$.  In particular, the vector 
$$\bfv  \quad = \quad \sum_{i =1}^n \bfu_i  \quad =  \quad  \sum_{i = 1}^n  e_{ii}   -  \sum_{i =1}^n e_{i-2,i} $$
is in the lattice $\ker_\zz A_n$.  This implies that the binomial 
$$  \prod_{i =1}^n  \sigma_{ii}  -  \prod_{i =1}^n  \sigma_{i-2,i}$$
belongs to the toric ideal $I_{A_n}$.  This implies that any  $\Sigma \in V(I_{A_n})$ satisfies 
$$\prod_{i = 1}^n  \sigma_{ii}^2  =  \prod_{i =1}^n  \sigma_{i-2,i}^2.$$
Thus, $V(I_{A_n})  \cap PD_n  = \emptyset$.
\end{proof}

As the last tool in the proof, we need to show that if $X$ satisfies some, but not all, of the conditional independence statements $X_i \ind X_{i+1} | X_{i +2}$, then none of the independence statements $X_i \ind X_{i+1}$ need to be satisfied.  This is explained in the following lemma.

\begin{lemma}\label{lem:noless}
Suppose that $X \sim \mathcal{N}(\mu, \Sigma)$ and $X$ satisfies all the cyclic conditional independence statements except for $X_{n-1} \ind X_{n} | X_1$.  Then $X$ need not satisfy any conditional independence statements of the form $X_i \ind X_{i+1}$.
\end{lemma}

\begin{proof}
It suffices to exhibit a positive definite covariance matrix that satisfies all the cyclic conditional independence constraints except $\sigma_{11} \sigma_{n-1,n} - \sigma_{1,n}  \sigma_{1,n-1}$ that does not have any zero entries.  To this end, let $0 < a, e <  \frac{1}{n}$.  Define $\sigma_{ii} = 1$ for all $i$, $\sigma_{i-1,i}  = a^{n-i+1}$ for $i  \in [n]$,  $\sigma_{i-2,i} = a$ for all $i \in [n]$, and $\sigma_{ij} = e$ for all other values of $i,j$.  Note that we take the indices cyclically modulo $n$, so $\sigma_{1,n} = a^n$.

First of all, note that the $\Sigma$ defined in this way is positive definite because it is diagonally dominant.  Furthermore, for all $i$ except $i = 1$, we have
$$\sigma_{ii}  \sigma_{i-1,i-2}  -  \sigma_{i-1,i} \sigma_{i,i-2}  \quad = \quad  1 \cdot a^{n-(i-1) + 1}  -  a^{n - i + 1}  \cdot a   \quad = \quad 0$$
and thus, $X$ satisfies all the cyclic conditional independence statements except $X_{n-1} \ind X_n | X_1$.  Finally, by construction, all elements of $\Sigma$ are nonzero.
\end{proof}

\noindent {\em Proof of Theorem \ref{thm:big}}:  Lemmas \ref{lem:minprime} and \ref{lem:badcomp} show that $V(I_n) \cap PD_n  =  V( \left< \sigma_{12},  \sigma_{23} \ldots, \sigma_{1n} \right>) \cap PD_n$
which implies the desired implication of conditional independence statements (reinterpretting those polynomial constraints back into CI constraints).  Lemma \ref{lem:noless} shows that if even one of the $n$ initial CI constraints is omitted, there exists positive definite covariance matrices that satisfy the indicated constraints and do not imply any of the marginal CI constraints $X_i \ind X_{i +1}$.  This completes the proof of the theorem.  \qed

%%%%%%%%%%%%%%%%%%%%%%%%%%%%%%%%
%%%%%%%%%%%%%%%%%%%%%%%%%%%%%%%%
%%%%%%%%%%%%%%%%%%%%%%%%%%%%%%%%
%%%%%%%%%%%%%%%%%%%%%%%%%%%%%%%%

\end{document}